\documentclass[12pt, a4paper]{amsart}
\usepackage{amsmath}
\usepackage{geometry,amsthm,graphics,tabularx,amssymb,
shapepar}
\usepackage{amscd}
\usepackage[usenames]{color}

\usepackage[all,2cell,dvips]{xy}

\newcommand{\CO}{{\mathcal {O}}}

\newcommand{\CS}{{\mathcal {S}}}

\newcommand{\cS}{{\mathcal {S}}}

\newcommand{\RH}{{\mathrm {H}}}

\newcommand{\RJ}{{\mathrm {J}}}

\newcommand{\RO}{{\mathrm {O}}}

\newcommand{\GL}{{\mathrm{GL}}}

\newcommand{\Hom}{{\mathrm{Hom}}}

\newcommand{\Ind}{{\mathrm{Ind}}}

\newcommand{\Sp}{{\mathrm{Sp}}}

\newcommand{\wt}{\widetilde}

\newcommand{\ov}{\overline}

\newcommand{\od}{\operatorname{d}}

\newcommand{\C}{\mathbb{C}}
\newcommand{\R}{\mathbb R}

\newcommand{\abs}[1]{\lvert#1\rvert}

\newcommand{\rj}{\mathrm{j}}

\newcommand{\be}{\begin {equation}}
\newcommand{\ee}{\end {equation}}
\newcommand{\bp}{\begin {proof}}
\newcommand{\ep}{\end {proof}}
\newcommand{\bee}{\begin {equation*}}
\newcommand{\eee}{\end {equation*}}

\theoremstyle{plain}
\newtheorem{thm}{Theorem}[section]
\newtheorem{lemma}[thm]{Lemma}

\newtheorem{prop}[thm]{Proposition}

\theoremstyle{definition}
\newtheorem{defi}[thm]{Definition}

\theoremstyle{remark}

\newtheorem{rmk}[thm]{Remark}

\begin{document}

\title[Full theta lifting for type II reductive dual pairs]{Full theta lifting for type II reductive dual pairs}

\author [H.Xue] {Huajian Xue}
\address{Department of Mathematics \\
Shantou Uninversity\\
Shantou\\
China} \email{hjxue@stu.edu.cn}



\begin{abstract}
In this article, we study the full theta lifting for two cases of type II reductive dual pairs over a nonarchimedean local field. Firstly, we determine the structure of the full theta lifts of all irreducible representations for dual pair $(\GL(2),\GL(2))$. Secondly, we prove that the full theta lift of an irreducible tempered representation is irreducible and tempered for dual pair $(\GL(n),\GL(n+1))$, where $n$ is a positive integer.
\end{abstract}

\maketitle

\section{Introduction}\label{sec-intro}

Let $F$ be a local field of characteristic 0. Let $W$ be a finite dimensional symplectic space over $F$ with symplectic form $\langle\, , \,\rangle_W$. Denote by $\RH(W):=W\times F$ the Heisenberg group attached to $W$.
The group multiplication is defined by
\[
(w_1, t_1)(w_2, t_2)=(w_1+w_2, t_1 +t_2+\langle w_1, w_2\rangle_W), \quad w_1, w_2\in W, t_1, t_2\in F.
\]
The symplectic group $\Sp(W)$ has a natural action on $\RH(W)$:
\[
g.(w, t):=(g.w, t), \quad g \in\Sp(W), (w, t)\in \RH(W).
\]
Let
\be \label{eq-cover}
1\to \{\pm 1\}\to \widetilde{\Sp}(W) \to \Sp(W)\to 1
\ee
be the metaplectic cover of $\Sp(W)$. The above exact sequence induces an action of  $\widetilde{\Sp}(W)$ on $\RH(W)$. We write  $\widetilde{\RJ}(W):=\widetilde{\Sp}(W)\ltimes \RH(W)$.

Fix a nontrivial unitary character $\psi$ of $F$. Then there is a unique (up to isomorphism) smooth representation $\omega$ of $\widetilde{\RJ}(W)$ such that $\omega|_{\RH(W)}$ is irreducible and has central character $\psi$.
The representation $\omega$ is called the Weil representation.

If $H$ is a closed subgroup of  $\Sp(W)$, we will write $\wt H$ for the preimage of $H$ in $\widetilde{\Sp}(W)$. Consider a reductive dual pair $(G, G')$ in  $\Sp(W)$. Note that $\wt{G}$ and $\wt{G'}$ commute with each other, and  the representation $\omega$
induces a representation of $\wt{G}\times\wt{G'}$, which we still denote by $\omega$ for simplicity if no confusion can arise. In the theory of local theta correspondence, one is interested in how $\omega$ decomposes into irreducible representations of $\wt{G}\times\wt{G'}$.

For a reductive group $G$, write $\mathrm{Irr}(G)$ for the set of isomorphism
classes of irreducible admissible representations of $G$. Let $\pi\in\mathrm{Irr}(\wt G)$. One may consider the maximal $\pi$-isotypic quotient
$\omega/\mathcal N(\pi)$
of $\omega$, where
\[
\mathcal N(\pi):=\bigcap_{\varphi\in \Hom_{\wt{G}}(\omega, \pi)} \ker(\varphi).
\]
There is a smooth admissible representation $\Theta(\pi)$ of $\wt{G'}$, such that
\be \label{decomp-weil}
\omega/\mathcal N(\pi)\cong \pi\widehat{\boxtimes} \Theta(\pi),
\ee
where ``$\widehat{\boxtimes}$'' stands for the completed tensor product (resp. algebraic tensor product) if $F$ is archimedean (resp. non-archimedean). We call $\Theta(\pi)$ the full theta lift of $\pi$. Note that $\Theta(\pi)$ may be zero, and $\Theta(\pi)=0$ if and only if $\Hom_{\wt{G}}(\omega, \pi)=0$.

\begin{thm} \label{theta-corresp.}
If $\Theta(\pi)\neq 0$, then there is a unique $\wt{G'}$ submodule $\Theta^0(\pi)$ of $\Theta(\pi)$, such that
\[
\theta(\pi):=\Theta(\pi)/\Theta^0(\pi)
\]
is irreducible.
\end{thm}
The representation $\theta(\pi)$ is called the (small) theta lift of $\pi$. The above theorem is known as Howe duality principle (see \cite{Ho} \cite{Wal} \cite{Mi}\cite{GT}\cite{GSu} for proofs of this theorem for various cases).

One basic problem is determining $\theta(\pi)$ and $\Theta(\pi)$ precisely for given $\pi\in\mathrm{Irr}(\wt G)$. Although there are many fruitful results concerning the (small) theta lift  $\theta(\pi)$ (see \cite{AB}\cite{LPTZ}\cite{Mo}\cite{Pa} etc.), few has been known about the full theta lift $\Theta(\pi)$. For applications to representation theory and automorphic forms, it is desirable to know whether or not the full theta lift itself is irreducible. If the local field $F$ is nonarchimedean, Kudla \cite{Ku} proves that $\Theta(\pi)$ has finite length. Furthermore, if $\pi$ is a supercuspidal representation, then $\Theta(\pi)$ (in the first occurrence) is irreducible and is also supercuspidal. In \cite{Mu}, Mui\'c shows that the lifts of discrete series behave very much like the lifts of supercuspidal representations. Explicitly, let $(G, G')$ be a type I reductive dual pair, and let $\pi$ be an irreducible discrete series representations of $G$. If the rank of $G'$ is in a suitable range (which we call the tempered range), then $\Theta(\pi)$ is irreducible and tempered. While outside this range, $\Theta(\pi)$ is either $0$ or non-tempered. In \cite{LM}, Loke and Ma prove that for type I dual pair in stable range over $\R$, the  full theta lift of a unitary representation is always irreducible except for one special case. In our previous work (see \cite{Xue}\cite{FSX}), we determined the full theta lifting explicitly for dual pair $(\GL(1),\GL(n))$, and proved that $\Theta(\pi)$ is irreducible if $\pi$ is generic for $(\GL(n),\GL(n))$.

In this article we will focus on theta lifting over nonarchimedean local field. So from now on, let $F$ be a nonarchimedean field of characteristic $0$. Denote by $\mathfrak{O}_F$ the ring of integers of $F$, and let $\varpi_F$ be a uniformizing element of it.
Write $q:=p^r$ for the cardinality of the residue field $\mathfrak{O}_F/\varpi_F\mathfrak{O}_F$, where $p$ is the characteristic of $\mathfrak{O}_F/\varpi_F\mathfrak{O}_F$.
Denote by $\abs{\,\cdot\,}_F$ the normalized absolute value on $F$. Here and henceforth, the subscript ``$F$''  in $\mathfrak{O}_F$, $\varpi_F$ and $\abs{\,\cdot\,} _F$ may be omitted
when no confusion can arise.

Let $n,m$ be two positive integers($n\leq m$). Then
$(G,G^\prime)=(\GL_n(F), \GL_m(F))$
is a type II reductive dual pair. Denote by $M_{n\times m}(F)$ the space of $n\times m$ matrices over $F$.
Let $\CS(M_{n\times m}(F))$ be the space of Schwartz-Bruhat functions on $M_{n\times m}(F)$,
namely,  the space of compactly supported, locally constant $\C$-valued functions on $M_{n\times m}(F)$.
The Weil representation can be realized on $\CS(M_{n\times m}(F))$. More precisely, note that $\CS(M_{n\times m}(F))$ carries a natural action of $\GL_n(F)\times \GL_m(F)$ which is given
by
\[
(g_1, g_2)f(M):=f(g_1^{-1}Mg_2), \quad (g_1, g_2)\in\GL_n(F)\times \GL_m(F), M\in M_{n\times m}(F).
\]
The Weil representation $\omega_{n,m}$ of $\GL_n(F)\times \GL_m(F)$ is defined to be
\[
\omega_{n,m}:=\CS(M_{n\times m}(F))\otimes(\abs\det^{-\frac{m}{2}}\boxtimes\abs\det^{\frac{n}{2}}).
\]
As usual, we do not distinguish a representation with its underlying space.

If $\pi\in \mathrm{Irr}(G)=\mathrm{Irr}(\GL_n(F))$, then the full theta lift of $\pi$ is given by
\[
\Theta(\pi)=\left(\omega_{n,m}\otimes\pi^\vee\right)_{G},
\]
where $V_G$ means taking the maximal quotient of $V$ on which $G$ (viewed naturally as a subgroup of $G\times G^\prime$) acts trivially.

In this paper we will study the structure of $\Theta(\pi)$ for two cases of type II dual pairs. We give an outline of this paper. Section \ref{sec-full-theta-(GL2-GL2)} is devoted to the full theta lifting for $(\GL_2(F),\GL_2(F))$. By checking the poles of the L-function of  $\pi$, we show that $\Theta(\pi)$ is irreducible if $\pi$ is non-trivial. The case of trivial representation is much more involved. In fact, we need to consider the generalized semi-invariant distribution on $M_{2\times 2}(F)$. We make use of some results about generalized semi-invariant distribution given in \cite{HS}, and show that $\Theta(\pi)$ is isomorphic to an induce representation of length 2 if $\pi$ is trivial. In Section \ref{sec-almost-eq-rank}, we study the dual pair $(\GL_n(F),\GL_{n+1}(F))$. We firstly introduce two main tools, Kudla's filtration and Casselman's square-integrability criterion, which will be frequently used in our proofs. Then we prove that the full theta lift of a discrete series representation is irreducible and tempered. Based on the discrete series case and an induction technique, we are able to show that the full theta lift of a tempered representation is also irreducible and tempered.

\textbf{Acknowledgements}: The author would like to thank Professor Binyong Sun for enlightening discussions and valuable suggestions. The author is supported by NSFC grant No. 12001350 and a startup grant from Shantou University (No. 35941912).

\section{Full theta lifting for $(\GL_2(F),\GL_2(F))$} \label{sec-full-theta-(GL2-GL2)}

\subsection{Irreducible representations of $\GL_2(F)$} \label{sec-rep-GL2}

In this subsection we put $G:=\GL_2(F)$, and denote by $B$  the standard Borel subgroup of $G$ which consists of upper triangular matrices. Let $\chi_1$ and $\chi_2$ be two characters of $F^\times$.
The pair $(\chi_1, \chi_2)$ determines a character $\chi$ of $B$ given by
\[
\chi\left( \begin{pmatrix}
a & \\
 & b
\end{pmatrix}\begin{pmatrix}
1 & x\\
 & 1
\end{pmatrix}\right):=\chi_1(a)\chi_2(b).
\]
We define the normalized smooth induction of $\chi$ to be
\[
I(\chi_1,\chi_2):=\Ind_B^G \chi=\left\{f\in C^\infty(G)\middle| f\left( \begin{pmatrix}
a & \\
 & b
\end{pmatrix}\begin{pmatrix}
1 & x\\
 & 1
\end{pmatrix}g\right)=\chi_1(a)\chi_2(b)\left|\frac{a}{b}\right|^{\frac{1}{2}}f(g)\right\},
\]
on which the group $G$ acts by right translation.

\begin{lemma}\label{lem-criterion-irr-Ind}(\cite[Theorem 3.3]{JL})
Let $\chi_1$ and $\chi_2$ be two characters of $F^\times$, then $I(\chi_1,\chi_2)$ is admissible. Moreover, \\
(i) if $\chi_1 \chi_2^{-1}\neq\abs{\,\cdot\,}^{\pm 1}$, then $I(\chi_1,\chi_2)$ is irreducible. \\
(ii) if $\chi_1 \chi_2^{-1}=\abs{\,\cdot\,}$, then $I(\chi_1,\chi_2)$ contains an irreducible subrepresentation of codimension $1$. \\
(iii) if $\chi_1 \chi_2^{-1}=\abs{\,\cdot\,}^{-1}$, then $I(\chi_1,\chi_2)$ contains a $1$-dimensional subrepresentation whose quotient is irreducible.
\end{lemma}

A representation $I(\chi_1,\chi_2)$ in case (i) of Lemma \ref{lem-criterion-irr-Ind} is called an irreducible principal series representation, while the irreducible subrepresentation of $I(\chi_1,\chi_2)$ in case (ii) and the irreducible quotient of $I(\chi_1,\chi_2)$ in case (iii) of Lemma \ref{lem-criterion-irr-Ind} are called special representations.

If $I(\chi_1,\chi_2)$ is irreducible, we write $\pi(\chi_1,\chi_2)=I(\chi_1,\chi_2)$. Otherwise, we define $\pi(\chi_1,\chi_2)$ to be the corresponding special representation.

For two pairs of characters $(\chi_1, \chi_2)$ and $(\tau_1, \tau_2)$, the representations $\pi(\chi_1,\chi_2)$ and $\pi(\tau_1,\tau_2)$ are isomorphic if and only if $(\chi_1, \chi_2)=(\tau_1,\tau_2)$ or $(\chi_1, \chi_2)=(\tau_2,\tau_1)$. In particular, for any special representation $\pi$ in case (ii) of Lemma \ref{lem-criterion-irr-Ind}, there is exactly one special representation $\pi^\prime$ in case (iii) such that
$\pi\cong\pi^\prime$.

The representation $\pi(\abs{\,\cdot\,}^{1/2},\abs{\,\cdot\,}^{-1/2})$ is called the Steinberg representation.

Denote by $\pi(\chi_1,\chi_2)^\vee$ the contragredient representation of $\pi(\chi_1,\chi_2)$. We have
\[
\pi(\chi_1,\chi_2)^\vee \cong \pi(\chi_1^{-1},\chi_2^{-1}).
\]

Also, it is not hard to show that
\[
\pi(\chi_1,\chi_2)\otimes (\chi\circ \det) \cong \pi(\chi\chi_1,\chi\chi_2)
\]
for any character $\chi$ of $F^\times$.

The  classification of irreducible representation of $G$ is given in the following lemma.

\begin{lemma} \label{classification} (\cite[Theorem 6.13.4]{GH})
Let $\pi$ be an irreducible admissible representation of $G=\GL_2(F)$. Then $\pi$ is isomorphic to one of the following disjoint types, where $\chi$, $\chi_1$ and $\chi_2$ are characters of $F^\times$:\\
(i) an irreducible principal series representation $\pi(\chi_1,\chi_2)$, where $\chi_1 \chi_2^{-1}\neq\abs{\,\cdot\,}^{\pm 1}$;\\
(ii) a special representation, which we may write in the form $\pi(\chi\abs{\,\cdot\,}^{1/2},\chi\abs{\,\cdot\,}^{-1/2})$;\\
(iii) a supercuspidal representation;\\
(iv) a one-dimensional representation, of the form $\chi \circ \det$.
\end{lemma}

\subsection{Theta lifting of nontrivial representations} \label{subsec-nontrivial rep}
The Godement-Jacquet L-functions for all irreducible representation $\pi$ of $\GL_2(F)$ are listed as follows (see \cite[Chapter I]{JL}).
\begin{itemize}
\item[(i)] For an irreducible principal series $\pi=\pi(\chi_1,\chi_2)$ (where $\chi_1 \chi_2^{-1}\neq\abs{\,\cdot\,}^{\pm 1}$),
\[
L(s,\pi)=\frac{1}{(1-\alpha_1q^{-s})(1-\alpha_2q^{-s})},
\]
where $\alpha_i=\chi_i(\varpi)$ if $\chi_i$ is unramified and $\alpha_i=0$ if $\chi_i$ is ramified.
\item[(ii)] For a special representation $\pi=\pi(\chi\abs{\,\cdot\,}^{1/2},\chi\abs{\,\cdot\,}^{-1/2})$,
\[
L(s,\pi)=\frac{1}{1-\alpha q^{-s}},
\]
where $\alpha=\chi(\varpi)\abs\varpi^{1/2}=q^{-1/2}\chi(\varpi)$ if $\chi$ is unramified and $\alpha=0$ if $\chi$ is ramified.
\item[(iii)] If $\pi$ is supercuspidal, then
\[
L(s,\pi)=1.
\]
\item[(iv)] For a one-dimensional representation $\pi=\chi\circ\det$,
\[\begin{split}
L(s,\pi)&=\frac{1}{(1-\chi(\varpi)\abs\varpi^{-1/2}q^{-s})(1-\chi(\varpi)\abs\varpi^{1/2}q^{-s})}\\
&= \frac{1}{(1-\chi(\varpi)q^{-s+1/2})(1-\chi(\varpi)q^{-s-1/2})}
\end{split}\]
if $\chi$ is unramified, and
\[
L(s,\pi)=1,
\]
if $\chi$ is ramified.
\end{itemize}

A direct calculation shows that if $\pi$ is not the trivial representation, then $L(s,\pi)$ has no pole at $s=1/2$, or $L(s,\pi^\vee)$ has no pole at $s=1/2$. Therefore, by \cite[Corollary 1.5]{FSX}, we have
\begin{thm}
If $\pi$ is not the trivial representation, then $\Theta(\pi)$ is irreducible.
\end{thm}

\subsection{Theta lifting of trivial representation} \label{subsec-trivial rep}

It remains to study the case that $\pi$ is a trivial representation.

Put $G=G^\prime=\GL_2(F)$. Let $\chi$ be a character of $G$.
For a $G\times G^\prime$-invariant locally closed subset $X$ of $M_2:=M_{2\times 2}(F)$, define the normalized $\chi$-coinvariant of $\CS(X)$ by
\[
\CS_\chi(X):=\left(\left(\CS(X)\otimes\left(\abs\det^{-1}\boxtimes\abs\det\right)\right)\otimes\chi^{-1}\right)_{G},
\]
recalling that $\CS(X)$ is the space of Schwartz-Bruhat functions on $X$, and $(\,\cdot\,)_{G}$ stands for the maximal quotient space which is invariant under the action of $G$. Then $\CS_\chi(X)$ carries an action of $G^\prime$. In particular,
\be
\CS_\chi(M_2)\cong \Theta(\chi).
\ee
By \cite[Propositon 1.8]{BZ1}, we have the following exact sequences
\be
\CS_\chi(\CO_1) \to \CS_\chi(\overline\CO_1) \to \CS_\chi(\{0\}) \to 0,
\ee
and
\be
\CS_\chi(\CO_2) \to \CS_\chi(M_2) \to \CS_\chi(\overline\CO_1) \to 0,
\ee
where $\CO_i$ is the subset of $M_2$ which consists of matrices with rank $i$ ($i=0,1,2$).
Write $\CS^*(X)$ for the set of tempered distributions on $X$, namely the space of continuous linear functionals on $\CS(X)$. Define the normalized $(G,\chi)$-invariant of $\CS^*(X)$ by
\[
\CS^*(X)^\chi:=\left(\left(\CS^*(X)\otimes\left(\abs\det\boxtimes\abs\det^{-1}\right)\right)\otimes\chi^{-1}\right)^{G},
\]
where $(\,\cdot\,)^{G}$ denotes the $G$-invariant subspace. Obviously,
\be \label{eqn-dual}
\CS^*(X)^\chi\cong \left(\CS_{\chi^{-1}}(X)\right)^*.
\ee
By duality, we have exact sequences
\be \label{exact-rank1}
0 \to \CS^*(\{0\})^\chi \to\CS^*(\overline\CO_1)^\chi \xrightarrow{\rj_1} \CS^*(\CO_1)^\chi,
\ee
and
\be \label{exact-rank2}
0 \to \CS^*(\overline\CO_1)^\chi \to\CS^*(M_2)^\chi \xrightarrow{\rj_2} \CS^*(\CO_2)^\chi.
\ee

\begin{defi}
Let $\chi$ be a character of a group $G$ and let $V$ be a (non-necessary smooth) representation of $G$. A vector $v\in V$ is called a generalized $(G,\chi)$-invariant (or simply generalized $\chi$-invariant) vector if there is a non-negative integer $k$ such that
\[
(g_0-\chi(g_0))(g_1-\chi(g_1))\cdots(g_k-\chi(g_k)).v=0
\]
for all $g_0,g_1,\cdots,g_k\in G$.
\end{defi}

Let $X$ be a $G\times G^\prime$-invariant locally closed subset of $M_2$.  Denote by $\CS^*(X)^{\chi,\infty}$ the space of generalized $(G,\chi)$-invariant vectors in
$$\CS^*(X)\otimes\left(\abs\det\boxtimes\abs\det^{-1}\right).$$
Apparently,
\[
\CS^*(X)^\chi \subset \CS^*(X)^{\chi,\infty}.
\]
We also have the following exact sequences
\be \label{exact-inf-rank1}
0 \to \CS^*(\{0\})^{\chi,\infty} \to\CS^*(\overline\CO_1)^{\chi,\infty} \xrightarrow{\rj_1^\prime} \CS^*(\CO_1)^{\chi,\infty},
\ee
and
\be \label{exact-inf-rank2}
0 \to \CS^*(\overline\CO_1)^{\chi,\infty} \to\CS^*(M_2)^{\chi,\infty} \xrightarrow{\rj_2^\prime} \CS^*(\CO_2)^{\chi,\infty}.
\ee

Note that $\CS^*(\{0\})=\C\delta_0$, where $\delta_0$ is the Dirac distribution. It is easy to check that $\CS^*(\{0\})^\chi=\CS^*(\{0\})^{\chi,\infty}\neq 0$ if and only if $\chi=\abs{\det}$.

From now on, we always set $\chi=1$ to be the trivial representation. The following lemma is an immediate consequence of \cite[Theorem 1.4]{HS}.
\begin{lemma} \label{j1-surjective}
The maps $\rj_1$ defined in the exact sequence \eqref{exact-rank1} and $\rj_1^\prime$ defined in \eqref{exact-inf-rank1} are isomorphisms.
\end{lemma}
\bp
By the previous paragraph,
$$\CS^*(\{0\})^\chi=\CS^*(\{0\})^{\chi,\infty}= 0.$$
Thus $\rj_1$ and $\rj_1^\prime$ are injective. On the other hand, \cite[Theorem 1.4]{HS} implies that the two maps are surjective. Therefore, the Lemma follows.
\ep

\begin{lemma} \label{rank1-nonzero}
The space $\CS^*(\CO_1)^\chi$ is nonzero.
\end{lemma}
\begin{proof}
Decompose $\CO_1$ into disjoint union of $G$-orbits as
\[
\CO_1= \bigcup_{a\in F} G \begin{pmatrix} 1 & a \\ 0 & 0 \end{pmatrix} \bigcup G\begin{pmatrix} 0 &1\\ 0 & 0 \end{pmatrix}.
\]
For each $(a,b)\neq (0,0)$, the stabilizer of $\begin{pmatrix} a & b \\ 0 & 0 \end{pmatrix}$ under the action of $G$ is
\[
G_1:=\left\{\begin{pmatrix} 1 & x \\ 0 & y \end{pmatrix}\middle\vert x\in F, y\in F^\times\right\}.
\]
Thus
\[
G/G_1 \simeq G \begin{pmatrix} a & b \\ 0 & 0 \end{pmatrix} \quad \mbox{for all }  (a,b)\neq (0,0).
\]
Let $\chi_0$ be a character of $G$.
It is well-known that the space
$$\Hom_{G}(\CS(G/G_1), \chi_0)\neq 0$$
if and only if $\chi_0\vert_{G_1} =\Delta_{G/G_1}$,
where $\Delta_{G/G_1}:=\frac{\Delta_{G}}{\Delta_{G_1}}$ is the ratio of modular characters of $G$ and $G_1$.
Also
\[
\Delta_{G/G_1}\begin{pmatrix} 1 & x \\ 0 & y \end{pmatrix}=\abs y \quad \mbox{for all } \begin{pmatrix} 1 & x \\ 0 & y \end{pmatrix}\in G_1.
\]
Thus $\Hom_{G}(\CS(G/G_1)\otimes \abs \det^{-1}, \chi)\neq 0$. Now the closeness of $G/G_1 \simeq G \begin{pmatrix} a & b \\ 0 & 0 \end{pmatrix}$ (where $(a,b)\neq (0,0)$) in
$\CO_1$ implies  that the space
\[
\begin{split}
\CS^*(\CO_1)^\chi &= \left(\left(\CS^*(\CO_1)\otimes\left(\abs\det\boxtimes\abs\det^{-1}\right)\right)\otimes\chi^{-1}\right)^{G} \\
& = \left(\CS^*(\CO_1)\otimes \abs\det\otimes\chi^{-1}\right)^{G} \boxtimes\abs\det^{-1} \\
& \cong \Hom_{G}\left(\CS(\CO_1)\otimes \abs \det^{-1}, \chi^{-1}\right)\boxtimes\abs\det^{-1}
\end{split}
\]
is nonzero.
\end{proof}

\begin{lemma} \label{lem-O_1-inf=O_1}
We have the following equality
\[
\CS^*(\CO_1)^\chi=\CS^*(\CO_1)^{\chi,\infty}.
\]
\end{lemma}
\begin{proof}
Note that
\be
\begin{split}
\CO_1&=\bigcup_{t\in F} G \begin{pmatrix} 1 & 0 \\ 0 & 0 \end{pmatrix}\begin{pmatrix} 1 & t \\ 0 & 1 \end{pmatrix} \bigcup \bigcup_{t\in F}G\begin{pmatrix} 0 &0\\ 0 & 1 \end{pmatrix}\begin{pmatrix} 1 & 0 \\ t & 1 \end{pmatrix} \\
&\simeq G/G_1 \times F \bigcup G/G_2 \times F,
\end{split}
\ee
where
\[
G_1:=\left\{\begin{pmatrix} 1 & x \\ 0 & y \end{pmatrix}\middle\vert x\in F, y\in F^\times\right\},
\]
and
\[
G_2:=\left\{\begin{pmatrix} y & 0 \\ x & 1 \end{pmatrix}\middle\vert x\in F, y\in F^\times\right\}.
\]
It suffices to prove that on both open subsets $G/G_1 \times F$ and $G/G_2 \times F$, every generalized $\chi$-invariant distribution is $\chi$-invariant.
Let $D_0$ be a nonzero element in
\[
\Hom_{G,\infty}((\CS(G/G_1)\boxtimes\CS(F))\otimes\left(\abs\det^{-1}\boxtimes\abs\det\right),\chi^{-1}):=\CS^*(G/G_1\times F)^{\chi,\infty}.
\]
Take a nonzero test vector $v\in\CS(F)$, and let $e_v$ be the evaluation map
\[
\begin{split}
e_v:&\quad\,\,\Hom_{G,\infty}((\CS(G/G_1)\boxtimes\CS(F))\otimes\left(\abs\det^{-1}\boxtimes\abs\det\right),\chi^{-1})\\
&\to\Hom_{\C}((\CS(G/G_1))\otimes\abs\det^{-1},\chi^{-1})
\end{split},
\]
\[
e_v(D)(\phi\otimes 1):=D(\phi\otimes v\otimes 1\otimes 1),\quad \phi\in\cS(G_1/H_1).
\]
Then
\[
e_v(D_0)\in\Hom_{G,\infty}((\CS(G/G_1))\otimes\abs\det^{-1},\chi^{-1})
:=\CS^*(G/G_1)^{\chi,\infty}.
\]
By Theorem 6.16 of \cite{HS},
\[
\CS^*(G/G_1)^{\chi,\infty}=\CS^*(G/G_1)^{\chi}.
\]
It follows that
\[
D_0 \in \CS^*(G/G_1\times F)^{\chi}.
\]
With the same technique, we can prove that  every generalized $\chi$-invariant distribution on the open subset
$G/G_2 \times F$ is $\chi$-invariant, which completes the proof for the lemma.
\end{proof}

Denote by $\rj$ the natural map from $\CS^*(M_2)$ into $\CS^*(\CO_2)$. Then $\rj_2$ defined in \eqref{exact-rank2} is induced from the map
\[
\rj\otimes 1 \otimes 1: \CS^*(M_2)\otimes\left(\abs\det\boxtimes\abs\det^{-1}\right) \to \CS^*(\CO_2)\otimes\left(\abs\det\boxtimes\abs\det^{-1}\right).
\]
Since $\CO_2=\GL_2(F)$, it is clear that $\dim\CS^*(\CO_2)^\chi=1$. Moreover,
$$D_0:=\chi^{-1}( g)\abs {\det g}\od\!g\otimes 1 \otimes 1$$
is a generator of this space, where $\od\!g$ is a Haar measure on $\CO_2=\GL_2(F)$.

\begin{lemma} \label{lem-D0-D1}
Let $D_0$ be the distribution defined above.
There is an element $D_1\in \CS^*(M_2)\otimes\left(\abs\det\boxtimes\abs\det^{-1}\right)$ such that
\[
(\rj\otimes 1 \otimes 1)(D_1)=D_0,
\]
and
\[
D_1\in\CS^*(M_2)^{\chi,\infty}\setminus\CS^*(M_2)^\chi.
\]
\end{lemma}

\begin{proof}

Write $\chi^\prime=\chi^{-1}$. Recall the zeta integral
\[
Z_{\chi^\prime}(s, \phi):=\int_G \phi(g){\chi^\prime}( g)\abs{\det g}^{s+\frac{1}{2}}\od\! g.
\]
The distribution $Z_{\chi^\prime}(s, \,\cdot\,)$ has a Laurent expansion
\[
Z_{\chi^\prime}(s, \,\cdot\,)=\sum_{i=-k}^{+\infty} Z_{{\chi^\prime},i}\left(1-q^{-s+\frac 1 2}\right)^i,
\]
where the $i$-th coefficient $Z_{{\chi^\prime},i}$ is a distribution with support containing in $\overline \CO_1$ if $i<0$. The constant term $Z_{{\chi^\prime},0}$ gives an extension to $\CS(M_2)$ of
the distribution $\chi^{-1}( g)\abs {\det g}\od\!g$ on $\CS(\CO_2)$. That is
\[
(\rj\otimes 1 \otimes 1)(Z_{{\chi^\prime},0}\otimes 1 \otimes 1)=D_0.
\]
By \cite[Proposition 5.20]{HS}, we have
\[
D_1:=Z_{{\chi^\prime},0}\otimes 1 \otimes 1 \in \CS^*(M_2)^{\chi,\infty}.
\]
We need to show that $D_1$ is not contained in $\CS^*(M_2)^\chi$.

Note that $\chi^\prime=\chi^{-1}= 1$. The L-function listed in the previous section shows that $\frac 1 2$ is a simple pole of $L(s, \chi^\prime)$. Thus
$ Z_{{\chi^\prime},i}=0$ if $i<-1$ and $ Z_{{\chi^\prime}, -1}\neq0$.

Let $g_0$ be an element in $G$ satisfying that $\det {g_0}=\varpi^{-1}$. Then
\[
g_0 .Z_{\chi^\prime}(s, \,\cdot\,)= q^{-s-\frac 1 2}Z_{\chi^\prime}(s, \,\cdot\,).
\]
Compare the Laurent expansions of the two sides of this equality term by term, we obtain
\[
g_0. Z_{{\chi^\prime}, 0} - \abs {\det{g_0}}^{-1} Z_{{\chi^\prime}, 0}= -q^{-1}Z_{{\chi^\prime}, -1}.
\]
Now define $D_1=Z_{{\chi^\prime}, 0}\otimes 1 \otimes 1 \in \CS^*(M_2)\otimes\left(\abs\det\boxtimes\abs\det^{-1}\right)$. Then
\[
\begin{split}
(\chi(g_0)-g_0)(Z_{{\chi^\prime}, 0}\otimes 1 \otimes 1)&=Z_{{\chi^\prime}, 0}\otimes 1 \otimes 1-g_0.\abs{\det{g_0}}Z_{{\chi^\prime}, 0}\otimes 1 \otimes 1\\
&=Z_{{\chi^\prime}, -1}\otimes 1 \otimes 1\neq0.
\end{split}
\]
Thus $D_1$ is not $(G,\chi)$-invariant, as is desired.
\end{proof}

\begin{prop} \label{prp-S(M2)-iso-S(O1)}
We have the following isomorphism
$$\CS^*(\CO_1)^\chi \cong \CS^*(M_2)^\chi.$$
\end{prop}

\begin{proof}
In view of Lemma \ref{j1-surjective} and \ref{rank1-nonzero}, it suffices to prove that the map $\rj_2$ defined in the exact sequence \eqref{exact-rank2} is zero.

By Theorem 1.5 of \cite{HS}, the map $\rj_2^\prime$ defined in \eqref{exact-inf-rank2} is surjective. By Lemma
\ref{j1-surjective} and Lemma \ref{lem-O_1-inf=O_1}, we have the following commutative diagram of exact sequences
\[
\begin{CD}
0 @>>> \CS^*(\CO_1)^\chi @>>>\CS^*(M_2)^\chi @>\rj_2>> \CS^*(\CO_2)^\chi @.\\
@. @|  @VVV @VVV @.\\
0 @>>> \CS^*(\CO_1)^{\chi,\infty} @>>>\CS^*(M_2)^{\chi,\infty} @>\rj_2^\prime>> \CS^*(\CO_2)^{\chi,\infty} @>>>0
\end{CD}
\]
Let $D_0, D_1$ be the distributions defined in Lemma \ref{lem-D0-D1}. Since $D_0$ is a nonzero element of the one dimensional space $\CS^*(\CO_2)^\chi$, we only need to show that there is no $D\in\CS^*(M_2)^\chi$ such that $\rj_2(D)=D_0$. Assume that $D_2$ is an element in $\CS^*(M_2)^\chi$ with $\rj_2(D_2)=D_0$. Also $\rj_2^\prime(D_2)=D_0$, and
\[
\rj_2^\prime(D_1-D_2)=0.
\]
Lemma \ref{lem-D0-D1} implies that
\[
D_1-D_2 \in \CS^*(M_2)^{\chi,\infty}\setminus\CS^*(M_2)^\chi.
\]
It follows that
\[
D_1-D_2 \in \CS^*(\CO_1)^{\chi,\infty}\setminus\CS^*(\CO_1)^\chi,
\]
which contradicts to Lemma \ref{lem-O_1-inf=O_1}. Therefore,
$$\CS^*(\CO_1)^\chi \cong \CS^*(M_2)^\chi.$$

\end{proof}

Now we are prepared to state the theorem about the full theta lift of the trivial representation.
\begin{thm} \label{trivial-rep}
If $\pi$ is the trivial representation of $\GL_2(F)$, then $\Theta(\pi)$ is isomorphic to an induce representation of length $2$. More precisely, $\Theta(\pi)$ has a unique irreducible subrepresentation which is isomorphic to the Steinberg representation, and the corresponding quotient is the trivial representation.
\end{thm}

\begin{proof}
Let $\pi=\chi=1$. By the duality \eqref{eqn-dual} and Proposition \ref{prp-S(M2)-iso-S(O1)}, one has that
\be \label{eqa-S(M2)-iso-S(O1)}
\Theta(\pi)\cong\cS_\chi(M_2)\cong\cS_\chi(\CO_1).
\ee

Note that
\[
\CO_1=G\begin{pmatrix} 0 & 0\\ 0 &1 \end{pmatrix} H \simeq ({F^\times(P_0\times Q_0)})\backslash({G\times H}),
\]
where $P_0$ (resp. $Q_0$) is the subgroup of $G$ (resp. $H$) of the form $\begin{pmatrix}* & 0 \\ * & 1\end{pmatrix}$ (resp. $\begin{pmatrix}* & * \\ 0 & 1\end{pmatrix}$), and $F^\times$ is view as a subgroup of $G\times H$ via the
embedding $a\mapsto \left(\begin{pmatrix}a & 0 \\ 0 & a\end{pmatrix},\begin{pmatrix}a^{-1} & 0 \\ 0 & a^{-1}\end{pmatrix}\right)$.

Define an $H$-intertwining map $\tau$ by
\[
\begin{array}{rcl}
\tau: \cS\left(({F^\times(P_0\times Q_0)})\backslash({G\times H})\right)\otimes\left(\abs\det^{-1}\boxtimes\abs\det\right) &\to& \cS(Q_0\backslash H) \otimes \abs \det, \\
\phi \otimes 1 \otimes 1 &\mapsto & \phi^\prime\otimes 1,
\end{array}
\]
where $\phi^\prime( Q_0g ):= \phi (1F^\times P_0, F^\times Q_0 g))$, and $H$ acts on $\cS\left(({F^\times(P_0\times Q_0)})\backslash({G\times H})\right)$ and $\cS(Q_0\backslash H)$ by right translation. Denote by $\cS(Q_0\backslash H)_{F^\times, \abs{\,\cdot\,}^{2}}$ the maximal quotient of $\cS(Q_0\backslash H)$ on which the subgroup $F^\times$ acts
through the character $\abs{\,\cdot\,}^{2}$. Let $\lambda$ be the natural map
\[
\lambda:\cS(Q_0\backslash H)\otimes \abs \det \to \cS(Q_0\backslash H)_{F^\times, \abs{\,\cdot\,}^{2}}\otimes \abs \det.
\]
Then the composition $\lambda \circ \tau$ induces an isomorphism

\be
\begin{split}
& \left(\cS\left(({F^\times(P_0\times Q_0)})\backslash({G\times H})\right)\otimes\left(\abs\det^{-1}\boxtimes\abs\det\right)\right)_{G} \\
\cong{}& \cS(Q_0\backslash H)_{F^\times, \abs{\,\cdot\,}^{2}}\otimes \abs \det \label{iso-O1-deg.prin-1}
\end{split}
\ee
of representations of $H$.

Moreover, it can be check that the linear map
\[
\begin{array}{rcl}
\cS(Q_0\backslash H) &\to& C^\infty(\GL_2(F)), \\
\phi & \mapsto & \left(g \mapsto  \int_{F^\times}\phi(Q_0 ga){\abs a}^2 \,\od^\times\! a\right)
\end{array}
\]
induces an $H$-intertwining isomorphism
\be
\cS(Q_0\backslash H)_{F^\times, \abs{\,\cdot\,}^{2}}\cong I\left(\abs{\,\cdot\,}^{-1/ 2}, \abs{\,\cdot\,}^{- 3/ 2}\right) \cong I\left(\abs{\,\cdot\,}^{ 1/ 2}, \abs{\,\cdot\,}^{- 1/ 2}\right)\otimes {\abs \det}^-1.
\ee
Equivalently,
\be \label{eqa-O1-iso-deg.prin-2}
\cS(Q_0\backslash H)_{F^\times, \abs{\,\cdot\,}^{2}}\otimes \abs \det \cong I\left(\abs{\,\cdot\,}^{ 1/ 2}, \abs{\,\cdot\,}^{- 1/ 2}\right).
\ee

By Lemma \ref{lem-criterion-irr-Ind}, the length of $I\left(\abs{\,\cdot\,}^{ 1/ 2}, \abs{\,\cdot\,}^{- 1 / 2}\right)$ is $2$, and it admits a unique irreducible subrepresentation, namely the Steinberg representation. The corresponding quotient is the trivial representation. Therefore, by the isomorphisms \eqref{eqa-S(M2)-iso-S(O1)}, \eqref{iso-O1-deg.prin-1} and \eqref{eqa-O1-iso-deg.prin-2}, the theorem is proved.
\end{proof}

\begin{rmk}
In fact, Theorem \ref{trivial-rep} gives a negative answer to Question 6.3 posed in \cite{APS}.
Let $\pi$ be an irreducible representation of $\GL_2(F)$. Recall the Euler-Poincar{\'e} characteristic \cite[Equation 3.4]{APS}
\[
\mathrm{EP}_{\GL_2(F)}(\omega, \pi):= \sum_i (-1)^i \mathrm{Ext}_{\GL_2(F)}^i(\omega, \pi),
\]
which is a well-defined element of the Grothendieck group of $\GL_2(F)$-modules.

By \cite[Theorem 5.13]{APS}, one has that
\be \label{eq-EP}
\mathrm{EP}_{\GL_2(F)}(\omega, \pi)_\infty = \pi,
\ee
where the subscript ``$\infty$'' stands for taking smooth vectors. Assume that $\omega$ is a projective $\GL_2(F)$-module, as is expected in \cite[Question 6.3]{APS}, then \eqref{eq-EP} implies that
\[
\Hom_{\GL_2(F)}(\omega, \pi)_\infty = \pi.
\]
On the other hand, Proposition 4.1 of \cite{APS} shows that
\[
\Hom_{\GL_2(F)}(\omega, \pi)_\infty = \Theta(\pi)^\vee.
\]
It follows that
\[
\Theta(\pi)^\vee= \pi,
\]
which contradicts to Theorem \ref{trivial-rep}.

\end{rmk}

\section{The almost equal rank case} \label{sec-almost-eq-rank}

In this section we study the full theta lifts of discrete series and tempered representations for dual pair $(\GL_n(F), \GL_{n+1}(F))$, where $F$ is a non-archimedean field with characteristic $0$.  The idea is due to Mui\'c \cite{Mu}, who has achieved fruitful results on the full theta lifting for type I dual pairs. See also \cite{GSa} for an exposition on the almost equal rank type I dual pairs.

\subsection{Kudla's filtration and Casselman's square-integrability criterion} \label{sec-Kud-and-Cas}

For a reductive group $G$, denote by $\mathrm{Rep}(G)$ the category of smooth representations of $G$. Let $P=MN$ be a parabolic subgroup of $G$, where $M$ and $N$ are the Levi and unipotent part of $P$ respectively. We have the normalized (smooth) induction functor
\[
\Ind_P^G: \mathrm{Rep}(M) \to \mathrm{Rep}(G).
\]
Moreover, we have the normalized Jacquet functor
\[
R_P: \mathrm{Rep}(G) \to \mathrm{Rep}(M).
\]
Let $\ov{P}=M\ov{N}$ be the opposite parabolic subgroup to $P$. We get a Jacquet functor $R_{\ov{P}}$ similarly. Then for $\pi\in \mathrm{Rep}(G)$, $\sigma\in\mathrm{Rep}(M)$, one has the standard Frobenius reciprocity
\be \label{eq-Frob}
\Hom_G(\pi, \Ind_P^G(\sigma)) \cong \Hom_M(R_P(\pi), \sigma),
\ee
and Bernstein's Frobenius reciprocity (see \cite[Theorem 20]{Ber})
\be \label{eq-Bern-Frob}
\Hom_G(\Ind_P^G(\sigma), \pi) \cong \Hom_M(\sigma, R_{\ov{P}}(\pi)).
\ee

When we apply the Jacquet functor to the Weil representation of $\GL_n(F)\times \GL_{n+1}(F)$, we should look into the Weil representation of $\GL_l(F)\times \GL_m(F)$ for various $l$ and $m$. To avoid confusion, we write $G_l$ for the first member of dual pair $(\GL_l(F), \GL_m(F))$, and write $G^\prime_m$ for the second member. Also, $P^\prime$, $M^\prime$ (and so on) stand for subgroups of  $G^\prime_m$.
Denote by $\omega_{l, m}$ the Weil representation of $G_l\times G^\prime_m$. Let $\pi$ be an irreducible representation of $G_l$. Denote by $\Theta_{l,m}(\pi)$ the full theta lift of $\pi$ to $G^\prime_m$. The unique irreducible quotient of $\Theta_{l,m}(\pi)$, i.e. the (small) theta lift of $\pi$, is denoted by $\theta_{l,m}(\pi)$.

\begin{lemma}(\cite[Lemma 2.4]{APS}) \label{lem-almost-lem1}
Let $\pi \in \mathrm{Irr}(G_n)$. We have the following isomorphism
\[
\Hom_{G_n}(\cS(G_n), \pi)_\infty \cong \pi
\]
as representation of $G^\prime_n$.
\end{lemma}

\begin{lemma}(\cite[Proposition 4.1]{APS}) \label{lem-almost-lem2}
Let $\pi \in \mathrm{Irr}(G_n)$. We have the following isomorphism
\[
\Hom_{G_n}(\omega_{n, m}, \pi)_\infty \cong \Theta_{n,m}(\pi)^\vee
\]
as representations of $G^\prime_m$.
\end{lemma}

Let $n=n_1+n_2+\cdots+n_k$ be a partition of $n$. We write $P_{n_1, n_2, \cdots, n_k}$ for the block-wise upper triangular parabolic subgroup of $G_n$ which has $M_{n_1, n_2, \cdots, n_k}:=G_{n_1}\times G_{n_2}\times\cdots\times G_{n_k}$ as a Levi factor.

The following lemma, known as the Kudla's filtration (for type II dual pairs), is due to M\'inguez \cite[Proposition 3.3]{Mi}.
\begin{lemma} \label{lem-kudla-filtration}
The normalized Jacquet module $R_{P^\prime_{t,n+1-t}}(\omega_{n, n+1})$ of $\omega_{n, n+1}$ has a filtration of $G_n\times M^\prime_{t,n+1-t}$ modules
\[
0=R^{t+1}\subset R^t \subset \cdots\subset R^1 \subset R^0=R_{P^\prime_{t,n+1-t}}(\omega_{n, n+1})
\]
with successive quotients
\[
\begin{split}
J^i:&= R^i/R^{i+1} \\
& \cong \Ind_{P_{i,n-i}\times P^\prime_{t-i,i}\times G^\prime_{n+1-t}}^{G_n\times M^\prime_{t,n+1-t}}(\chi_{t,i} \otimes \cS(G_i)\otimes \omega_{n-i, n+1-t}).
\end{split}
\]
Here, $\chi_{t,i}$ is a character given by
\[
\chi_{t,i}=
\left\{ \begin{array}{ll}
\abs\det ^{\frac{-1+2t-i}{2}} \quad &\textrm{ on } G_i \\
\abs\det ^{\frac{t-i}{2}} \quad &\textrm{ on } G_{n-i} \\
\abs\det ^{\frac{1-i+t}{2}} \quad  &\textrm{ on } G^\prime_{t-i}\\
\abs\det ^{\frac{1-2t+i}{2}} \quad &\textrm{ on } G^\prime_{i} \\
\abs\det ^{\frac{-t+i}{2}}  \quad &\textrm{ on } G^\prime_{n+1-t}.
  \end{array}
  \right.
\]
While $G_i$ and $G_i^\prime$ act on $\cS(G_i)$ via left and right translation. If $i>t$, then $J^i$ is interpreted to be $0$.
\end{lemma}

In particular, the bottom piece
\[
J^t=R^t\cong \Ind_{P_{t,n-t}\times G^\prime_{t}\times G^\prime_{n+1-t}}^{G_n\times M^\prime_{t,n+1-t}}(\chi_{t,t} \otimes \cS(G_t)\otimes \omega_{n-t, n+1-t}),
\]
where
\[
\chi_{t,t}=
\left\{ \begin{array}{ll}
\abs\det ^{\frac{-1+t}{2}} \quad &\textrm{ on } G_t \\
1 \quad &\textrm{ on } G_{n-t} \\
\abs\det ^{\frac{1-t}{2}} \quad &\textrm{ on } G^\prime_{t} \\
1  \quad &\textrm{ on } G^\prime_{n+1-t},
  \end{array}
  \right.
\]
is a subrepresentation of $R_{P^\prime_{t,n+1-t}}(\omega_{n, n+1})$. This subrepresentation plays an important role in our induction procedure.

The following lemma is a direct consequence of Casselman square-integrability criterion (see \cite[Theorem 6.5.1]{Ca2}). We shall frequently exploit it.
\begin{lemma} \label{lem-casselman}
Let $\pi$ be a discrete series representation (resp. tempered representation) of $G_n$. Then on each irreducible constituents of $R_{P_{t,n-t}}(\pi)$, the center of $G_t$ acts by a character of the form $\chi\abs{\,\cdot\,}^\alpha$, where $\chi$ is unitary and $\alpha>0$ (resp. $\alpha\geq0$).
\end{lemma}

\subsection{The discrete series representations} \label{sec-alm-disc}

\begin{prop} \label{prop-discr-temp}
Let $\pi$ be an irreducible discrete series representation of $G_n$. Then every irreducible subquotient of $\Theta_{n,n+1}(\pi)$ is a tempered representation of  $G^\prime_{n+1}$.
\end{prop}

\begin{proof}
Fix a $G^\prime_{n+1}$-invariant filtration of $\Theta_{n,n+1}(\pi)$:
\[
0=\Sigma_{\ell+1}\subset \Sigma_\ell \subset \cdots \subset \Sigma_1\subset \Sigma_0=\Theta_{n,n+1}(\pi)
\]
such that
\[
\sigma_i:=\Sigma_i/\Sigma_{i+1}
\]
is irreducible for all $i$. We argue by contradiction that all $\sigma_i$ are tempered representations.  Assume that $k$ is the smallest integer such that $\sigma_k$ is a non-tempered representation. Then there is parabolic subgroup $P^\prime_{t, n+1-t}$ of $G^\prime_{n+1}$ such that
\[
\sigma_k \hookrightarrow \Ind_{P^\prime_{t, n+1-t}}^{G^\prime_{n+1}}(\tau\abs\det^{-s_0}\boxtimes \rho_0),
\]
where $\tau$ is an irreducible discrete series representation of $G_t^\prime$, $s_0>0$, and $\rho_0$ is an irreducible representation of $G^\prime_{n+1-t}$. By Frobenius reciprocity \eqref{eq-Frob}, we have
\[
R_{P^\prime_{t, n+1-t}}(\sigma_k) \twoheadrightarrow \tau\abs\det^{-s_0}\boxtimes \rho_0.
\]
By the  exactness of the Jacquet functor, we have
\[
R_{P^\prime_{t, n+1-t}}(\Sigma_{\ell+1})\subset R_{P^\prime_{t, n+1-t}}(\Sigma_\ell) \subset \cdots \subset R_{P^\prime_{t, n+1-t}}( \Sigma_0)
\]
with
\[
R_{P^\prime_{t, n+1-t}}(\sigma_i)=R_{P^\prime_{t, n+1-t}}(\Sigma_i)/R_{P^\prime_{t, n+1-t}}(\Sigma_{i+1}).
\]
Suppose that
\[
\begin{split}
\tau\abs\det^{-s_0}\boxtimes \rho_0&=(R_{P^\prime_{t, n+1-t}}(\Sigma_k)/R_{P^\prime_{t, n+1-t}}(\Sigma_{k+1}))/(T/R_{P^\prime_{t, n+1-t}}(\Sigma_{k+1}))\\
&=R_{P^\prime_{t, n+1-t}}(\Sigma_k)/T,
\end{split}
\]
then we obtain the following exact sequence of  $M^\prime_{t, n+1-t}$-modules:
\be \label{eq-exact-seq-split}
0\to \tau\abs\det^{-s_0}\boxtimes \rho_0 \to R_{P^\prime_{t, n+1-t}}(\Theta_{n,n+1}(\pi))/T\to A\to 0,
\ee
where $A=R_{P^\prime_{t, n+1-t}}(\Theta_{n,n+1}(\pi))/R_{P^\prime_{t, n+1-t}}(\Sigma_k)$ is a finite length representation which has a filtration with successive quotients $R_{P^\prime_{t, n+1-t}}(\sigma_i)$ for $i<k$. A key observation is  that the above exact sequence splits. Indeed, since $\sigma_i$ is a tempered representation for $i<k$, Lemma \ref{lem-casselman} asserts that on each irreducible constituent of $R_{P^\prime_{t, n+1-t}}(\sigma_i)$, the center of $G^\prime_t$ acts by a character of the form $\chi\abs{\,\cdot\,}^\alpha$ with $\chi$  unitary and $\alpha\geq0$. On the other hand,  the center of $G^\prime_t$ acts on $\tau\abs\det^{-s_0}\boxtimes \rho_0$ by $\abs{\,\cdot\,}^{-s_0 t}$ up to a unitary  character. It follows that the sequence \eqref{eq-exact-seq-split} splits.

Thus one has a nonzero $M^\prime_{t, n+1-t}$-equivariant map
\[
R_{P^\prime_{t, n+1-t}}(\Theta_{n,n+1}(\pi)) \to \tau\abs\det^{-s_0}\boxtimes \rho_0.
\]
By Frobenius reciprocity, one has a nonzero $G_n\times G^\prime_{ n+1}$-equivariant map
\[
\omega_{n,n+1} \twoheadrightarrow \pi\boxtimes\Theta_{n,n+1}(\pi) \to \pi \boxtimes \Ind_{P^\prime_{t, n+1-t}}^{G^\prime_{ n+1}}(\tau\abs\det^{-s_0}\boxtimes \rho_0).
\]
Hence we have
\[
\begin{split}
\pi^\vee &\hookrightarrow \Hom_{G^\prime_{ n+1}} (\omega_{n,n+1}, \Ind_{P^\prime_{t, n+1-t}}^{G^\prime_{ n+1}}(\tau\abs\det^{-s_0}\boxtimes \rho_0))_\infty \\
&=\Hom_{M^\prime_{t, n+1-t}}(R_{P^\prime_{t, n+1-t}}(\omega_{n,n+1}), \tau\abs\det^{-s_0}\boxtimes \rho_0)_\infty.
\end{split}
\]
Recall that the subscript $``\infty''$ stands for taking smooth vectors. By Lemma \ref{lem-kudla-filtration}, we have a natural restriction map
\[
\Hom_{M^\prime_{t, n+1-t}}(R_{P^\prime_{t, n+1-t}}(\omega_{n,n+1}), \tau\abs\det^{-s_0}\boxtimes \rho_0)_\infty
\to \Hom_{M^\prime_{t, n+1-t}}(J^t, \tau\abs\det^{-s_0}\boxtimes \rho_0)_\infty.
\]
We claim that the composition of the above two maps gives a nonzero $G_n$-equivariant map
\[
\pi^\vee \hookrightarrow \Hom_{M^\prime_{t, n+1-t}}(J^t, \tau\abs\det^{-s_0}\boxtimes \rho_0)_\infty.
\]

We use Casselman's square-integrability criterion Lemma \ref{lem-casselman} to argue for this claim. Assume that it is not true, then there would exist a nonzero $G_n$-equivariant map
\[
\pi^\vee \hookrightarrow \Hom_{M^\prime_{t, n+1-t}}(J^i, \tau\abs\det^{-s_0}\boxtimes \rho_0)
\]
for some $i<t$. In particular, the latter space would be nonzero for some $i<t$. By Lemma \ref{lem-kudla-filtration} and \eqref{eq-Bern-Frob}, the space
\[
\Hom_{M^\prime_{t-i,i}\times G^\prime_{n+1-t}}(\chi_{t,i} \otimes \cS(G_i)\otimes \omega_{n-i, n+1-t}, R_{\ov{P^\prime_{t-i,i}}}(\tau)\abs\det^{-s_0}\boxtimes \rho_0)
\]
would be nonzero. Since $\tau$ is an irreducible discrete series representation, Casselman's square-integrability criterion implies that
\[
R_{\ov{P^\prime_{t-i,i}}}(\tau)=\tau_1\abs\det^{t_1}\boxtimes \tau_2\abs\det^{t_2}
\]
for some irreducible discrete series representations $\tau_1$ and $\tau_2$ of $G^\prime_{t-i}$ and $G^\prime_{i}$ respectively, and $t_1, t_2\in\R$ such that
\[
t_1<t_2 \quad \textrm{and} \quad t_1\,\cdot\,(t-i)+t_2\,\cdot\,i=0.
\]
In particular, we have $t_1\leq 0$. Hence, the center of $G^\prime_{t-i}$ acts on $R_{\ov{P^\prime_{t-i,i}}}(\tau)\abs\det^{-s_0}$ by $\abs\det^{t_1-s_0}$ up to a unitary character, while $G^\prime_{t-i}$ acts on $\chi_{t,i} \otimes \cS(G_i)\otimes \omega_{n-i, n+1-t}$ by $\abs\det ^{\frac{1-i+t}{2}}$ (see Lemma \ref{lem-kudla-filtration}). But
\[
t_1-s_0<0 \quad \textrm{and} \quad 1-i+t> 0,
\]
which is a contradiction. The claim is proved.

Note that
\[
\begin{split}
&\Hom_{M^\prime_{t, n+1-t}}(J^t, \tau\abs\det^{-s_0}\boxtimes \rho_0)_\infty \\
= & \Hom_{M^\prime_{t, n+1-t}}(\Ind_{P_{t,n-t}\times G^\prime_{t}\times G^\prime_{n+1-t}}^{G_n\times M^\prime_{t,n+1-t}}(\chi_{t,t} \otimes \cS(G_t)\otimes \omega_{n-t, n+1-t}), \tau\abs\det^{-s_0}\boxtimes \rho_0)_\infty\\
\cong & \Ind_{P_{t,n-t}}^{G_n}(\tau\abs\det^{-s_0}\boxtimes \Hom_{G^\prime_{n+1-t}}(\omega_{n-t, n+1-t},\rho_0)_\infty).
\end{split}
\]
Here we apply Lemma \ref{lem-almost-lem1}  in the last isomorphism.
Thus we obtain a nonzero $G_n$-equivariant map
\[
\pi^\vee \hookrightarrow \Ind_{P_{t,n-t}}^{G_n}(\tau\abs\det^{-s_0}\boxtimes \Hom_{G^\prime_{n+1-t}}(\omega_{n-t, n+1-t},\rho_0)_\infty).
\]
Again by Frobenius reciprocity,  there is a nonzero $M_{t,n-t}$-equivariant map
\[
R_{M_{t,n-t}}(\pi^\vee) \to \tau\abs\det^{-s_0}\boxtimes \pi_0
\]
for some irreducible representation $\pi_0$ of $G_{n-t}$. Since $\pi^\vee$ is a discrete series representation, Lemma \ref{lem-casselman} implies that $s_0<0$, which is a contradiction. Therefore we complete the proof.
\end{proof}

\begin{prop} \label{prop-nplus1-to-n}
Let $\pi$ be an irreducible discrete series representation of $G_n$. If $\sigma$ is an irreducible subquotient of $\Theta_{n,n+1}(\pi)$, then we have $\sigma\subset \Ind_{P^\prime_{1,n}}^{G^\prime_{n+1}}(1\boxtimes \sigma_0)$ for some irreducible constituent $\sigma_0$ of $\Theta_{n,n}(\pi)$. Here $``1''$ stands for the trivial character of $G^\prime_1$.
\end{prop}

\begin{proof}
By the explicit theta correspondence for type II dual pair (see \cite[Theorem 1]{Mi}), we have a nonzero $G^\prime_{n+1}$-equivariant map
$$\theta_{n,n+1}(\pi) \to \Ind_{P^\prime_{1,n}}^{G^\prime_{n+1}}(1\boxtimes \theta_{n,n}(\pi)).$$
Let $\sigma$ be an irreducible subquotient of $\Theta_{n,n+1}(\pi)$. Since all irreducible subquotients of $\Theta_{n,n+1}(\pi)$ have the same supercuspidal support (see \cite[Lemma 4.2]{Mu}), we have $\sigma \subset \Ind_{P^\prime_{1,n}}^{G^\prime_{n+1}}(1\boxtimes \sigma_0)$ for some irreducible representation $\sigma_0$ of
$G^\prime_{n}$. Denote by $R_{P^\prime_{1,n}}(\sigma)_{(G^\prime_1, 1)}$ the maximal $(G^\prime_1, 1)$-isotypic quotient of $R_{P^\prime_{1,n}}(\sigma)$, then $R_{P^\prime_{1,n}}(\sigma)_{(G^\prime_1, 1)}$ has $1\boxtimes \sigma_0$ as a quotient. It follows that there is a surjective map
\[
\begin{split}
R_{P^\prime_{1,n}}(\omega_{n,n+1}) &\twoheadrightarrow \pi \boxtimes R_{P^\prime_{1,n}}(\Theta_{n,n+1}(\pi)) \\
& \twoheadrightarrow \pi \boxtimes R_{P^\prime_{1,n}}(\Theta_{n,n+1}(\pi))_{(G^\prime_1, 1)}
\end{split}
\]
and the latter has $\pi\boxtimes 1\boxtimes \sigma_0$ as a subquotient. We claim that the above map vanishes when restrict to $J^1$.

 Suppose that
\[
\Hom_{G_n \times M^\prime_{1,n}}(J^1, \pi \boxtimes R_{P^\prime_{1,n}}(\Theta_{n,n+1}(\pi))_{(G^\prime_1, 1)})
\]
is nonzero. Then there exists a nonzero $G_n \times G^\prime_1$-equivariant map
\[
J^1 \to \pi\boxtimes 1.
\]
We set $t=i=1$ in Lemma \ref{lem-kudla-filtration} to obtain
\[
J^1=\Ind_{P_{1,n-1}\times G_1^\prime\times G_n^\prime}^{G_{n}\times M_{1,n}^\prime}(\chi_{1,1}\otimes \cS(G_1)\otimes \omega_{n-1,n}),
\]
where $\chi_{1,1}=1$ is the trivial character. Thus
\[
\begin{split}
\Hom_{G_n \times G^\prime_1}(J^1, \pi\boxtimes 1)&=\Hom_{G_n \times G^\prime_1}(\Ind_{P_{1,n-1}\times G_1^\prime\times G_n^\prime}^{G_{n}\times M_{1,n}^\prime}(\chi_{1,1}\otimes \cS(G_1)\otimes \omega_{n-1,n}), \pi\boxtimes 1) \\
&\cong\Hom_{M_{1,n-1}}(1\boxtimes \omega_{n-1,n}, R_{\ov{P_{1,n-1}}}(\pi)).
\end{split}
\]
Here we apply Lemma \ref{lem-almost-lem1}  in the last isomorphism.
However, Casselman's square-integrability criterion implies that $G_1$ acts on $R_{\ov{P_{1,n-1}}}(\pi)$ by $\abs{\,\cdot\,}^s$ with $s<0$ up to a unitary character, which implies that
$$\Hom_{M_{1,n-1}}(1\boxtimes \omega_{n-1,n}, R_{\ov{P_{1,n-1}}}(\pi))= 0,$$
and we get a contradiction. Hence the claim holds. On the other hand,
by Lemma \ref{lem-kudla-filtration}, there is an exact sequence
\[
0 \to J^1 \to R_{P^\prime_{1,n}}(\omega_{n,n+1}) \to 1\boxtimes \omega_{n,n} \to 0.
\]
It follows by the claim that there is an $G^\prime_n$-equivariant surjective map
\be \label{eq-nplus1-to-n}
\Theta_{n,n}(\pi) \twoheadrightarrow R_{P^\prime_{1,n}}(\Theta_{n,n+1}(\pi))_{(G^\prime_1, 1)}.
\ee
Thus $\sigma_0$ is a subquotient of $\Theta_{n,n}(\pi)$.
\end{proof}

\begin{thm} \label{thm-alm-eq-rank-discr}
Let $\pi$ be an irreducible discrete series representation of $G_n$, then $\Theta_{n,n+1}(\pi)$ is an irreducible  tempered representation of  $G^\prime_{n+1}$.
\end{thm}

\begin{proof}
It suffices to show the  irreducibility. We know that $\Theta_{n,n}(\pi)$ is irreducible by  \cite[Theorem 1.7]{FSX}.
The Frobenius reciprocity and Proposition \ref{prop-nplus1-to-n} imply that, for any irreducible subquotient $\sigma$ of $\Theta_{n,n+1}(\pi)$,
the representation $R_{P^\prime_{1,n}}(\sigma)$ has $1\boxtimes \Theta_{n,n}(\pi)$ as a quotient. To prove that $\Theta_{n,n+1}(\pi)$ is irreducible, it suffices to show that $1\boxtimes \Theta_{n,n}(\pi)$ occurs in $R_{P^\prime_{1,n}}(\Theta_{n,n+1}(\pi))$ with multiplicity one. But this follows from \eqref{eq-nplus1-to-n}.
\end{proof}
\subsection{The tempered representations} \label{sec-alm-temp}

Analogous to Lemma \ref{lem-kudla-filtration}, there is a filtration for the Jacquet module $R_{P_{k,n-k}}(\omega_{n, n+1})$ of $\omega_{n, n+1}$ (see \cite[Proposition 3.2]{Mi}).
\begin{lemma} \label{lem-kudla-filtration2}
The normalized Jacquet module $R_{P_{k,n-k}}(\omega_{n, n+1})$ of $\omega_{n, n+1}$ has a filtration of $M_{k,n-k}\times G^\prime_{n+1}$ modules
\[
0=S^{k+1}\subset S^k \subset \cdots\subset S^1 \subset S^0=R_{P_{k,n-k}}(\omega_{n, n+1})
\]
with successive quotients
\[
\begin{split}
L^i:&= S^i/S^{i+1} \\
& \cong \Ind_{P_{k-i,i}\times G_{n-k}\times P^\prime_{i,n+1-i}}^{M_{k,n-k}\times G^\prime_{n+1}}(\xi_{k,i} \otimes \cS(G_i) \otimes\omega_{n-k, n+1-i} ).
\end{split}
\]
Here, $\xi_{k,i}$ is a character given by
\[
\xi_{k,i}=
\left\{ \begin{array}{ll}
\abs\det ^{\frac{1+k-i}{2}} \quad &\textrm{ on } G_{k-i} \\
\abs\det ^{\frac{1+2k-i}{2}} \quad &\textrm{ on } G_{i} \\
\abs\det ^{\frac{k-i}{2}} \quad  &\textrm{ on } G_{n-k}\\
\abs\det ^{\frac{-1-2k+i}{2}} \quad &\textrm{ on } G^\prime_{i} \\
\abs\det ^{\frac{-k+i}{2}}  \quad &\textrm{ on } G^\prime_{n+1-i}.
  \end{array}
  \right.
\]
While $G_i$ and $G_i^\prime$ act on $\cS(G_i)$ via left and right translation. If $i>k$, then $L^i$ is interpreted to be $0$.
\end{lemma}

Make use of Casselman's square-integrability criterion, we obtain the following lemma.
\begin{lemma} \label{lem-hom-j^i}
Let $\pi_0 \in \mathrm{Irr}(G_{n-k})$, and $\tau$ be an irreducible discrete series representation of $G_k$, then
\[
\Hom_{M_{k, n-k}} (L^i, \tau \boxtimes \pi_0)
\]
is nonzero if and only if $i=k$, in which case
\[
\Hom_{M_{k, n-k}} (L^k, \tau \boxtimes \pi_0)_\infty \cong \Ind_{P^\prime_{k,n+1-k}}^{G^\prime_{n+1}}(\tau \boxtimes \Theta_{n-k,n+1-k}(\pi_0)^\vee).
\]
\end{lemma}
\bp
The proof is analogous to that for Proposition \ref{prop-discr-temp}. Assume that
\[
\Hom_{M_{k, n-k}} (L^i, \tau \boxtimes \pi_0)\neq 0
\]
for some $i<k$. Then by Lemma \ref{lem-kudla-filtration2} and \eqref{eq-Bern-Frob}, the space
\[
\Hom_{M_{k-i,i}\times G_{n-k}}(\xi_{k,i} \otimes \cS(G_i)\otimes \omega_{n-k, n+1-i}, R_{\ov{P_{k-i,i}}}(\tau)\boxtimes \pi_0)
\]
would be nonzero. Since $\tau$ is an irreducible discrete series representation, Casselman's square-integrability criterion implies that
\[
R_{\ov{P_{k-i,i}}}(\tau)=\tau_1\abs\det^{t_1}\boxtimes \tau_2\abs\det^{t_2}
\]
for some irreducible discrete series representations $\tau_1$ and $\tau_2$ of $G_{k-i}$ and $G_{i}$ respectively, and $t_1, t_2\in\R$ such that
\[
t_1<t_2 \quad \textrm{and} \quad t_1\,\cdot\,(k-i)+t_2\,\cdot\,i=0.
\]
In particular, we have $t_1\leq 0$. Hence, the center of $G_{k-i}$ acts on $R_{\ov{P_{k-i,i}}}(\tau)$ by $\abs\det^{t_1}$ up to a unitary character, while $G_{k-i}$ acts on $\xi_{k,i} \otimes \cS(G_i)\otimes \omega_{n-k, n+1-i}$ by $\abs\det ^{\frac{1+k-i}{2}}$ (see Lemma \ref{lem-kudla-filtration2}). But
\[
t_1\leq0 \quad \textrm{and} \quad 1+k-i> 0,
\]
which is a contradiction. Thus $\Hom_{M_{k, n-k}} (L^i, \tau \boxtimes \pi_0)= 0$ for all $i<k$.

By Lemma \ref{lem-almost-lem1} and Lemma \ref{lem-almost-lem2}, we have
\[
\Hom_{G_k}(\cS(G_k), \tau)_\infty\cong \tau,
\]
and
\[
\Hom_{G_{n-k}}(\omega_{n-k, n+1-k}, \pi_0)_\infty \cong \Theta_{n-k,n+1-k}(\pi_0)^\vee.
\]
It follows that
\[
\begin{split}
&\Hom_{M_{k, n-k}}(L^k, \tau\boxtimes \pi_0)_\infty \\
= & \Hom_{M_{k, n-k}}(\Ind_{G_k\times G_{n-k}\times P^\prime_{k,n+1-k}}^{M_{k,n-k}\times G^\prime_{n+1}}(\xi_{k,k} \otimes \cS(G_k)\otimes \omega_{n-k, n+1-k}), \tau\boxtimes \pi_0)_\infty\\
\cong & \Ind_{P^\prime_{k,n+1-k}}^{G^\prime_{n+1}}(\tau \boxtimes \Theta_{n-k,n+1-k}(\pi_0)^\vee).
\end{split}
\]
Hence we complete the proof.
\ep

\begin{lemma} \label{lem-induction-prin}
Let $\pi \in \mathrm{Irr}(G_n)$, $\pi_0 \in \mathrm{Irr}(G_{n-k})$, and let $\tau$ be an irreducible discrete series representation of $G_k$. If $\pi$ is a subrepresentation of $\Ind_{P_{k,n-k}}^{G_n}(\tau \boxtimes \pi_0)$, then $\Theta_{n,n+1}(\pi)^\vee$ is a subrepresentation of $\Ind_{P^\prime_{k,n+1-k}}^{G^\prime_{n+1}}(\tau \boxtimes \Theta_{n-k,n+1-k}(\pi_0)^\vee)$.
\end{lemma}

\begin{proof}
By Lemma \ref{lem-almost-lem2}, we have
\[
\begin{split}
\Theta_{n,n+1}(\pi)^\vee & \cong \Hom_{G_n}(\omega_{n,n+1}, \pi)_\infty \\
& \hookrightarrow \Hom_{G_n}(\omega_{n,n+1}, \Ind_{P_{k,n-k}}^{G_n}(\tau \boxtimes \pi_0))_\infty \\
& \cong \Hom_{M_{k, n-k}} (R_{P_{k, n-k}}(\omega_{n,n+1}), \tau \boxtimes \pi_0))_\infty.
\end{split}
\]
Moreover, by Lemma \ref{lem-kudla-filtration2} and Lemma \ref{lem-hom-j^i}, we have
\[
\begin{split}
&\Hom_{M_{k, n-k}} (R_{P_{k, n-k}}(\omega_{n,n+1}), \tau \boxtimes \pi_0))_\infty \\
= & \Hom_{M_{k, n-k}} (L^k, \tau \boxtimes \pi_0))_\infty \\
\cong & \Ind_{P^\prime_{k,n+1-k}}^{G^\prime_{n+1}}(\tau \boxtimes \Theta_{n-k,n+1-k}(\pi_0)^\vee).
\end{split}
\]
Then the lemma follows.
\end{proof}

Now we come to the main result in this section.
\begin{thm} \label{thm-alm-eq-rank-temp}
Let $\pi$ be an irreducible tempered representation of $G_n$. Then $\Theta_{n,n+1}(\pi)$ is irreducible and  tempered.
\end{thm}

\begin{proof}
We prove by induction on $n$. If $\pi$ is a discrete series representation, then the assertion has been proved in Theorem \ref{thm-alm-eq-rank-discr}.

Suppose that $\pi$ is tempered but not a discrete series. Then there exists an irreducible discrete representation $\tau$ of $G_k$, and an irreducible tempered representation $\pi_0$ of $G_{n-k}$, such that
\[
\pi \hookrightarrow \Ind_{P_{k,n-k}}^{G_n}(\tau \boxtimes \pi_0).
\]
It follows from Lemma \ref{lem-induction-prin} that
\[
\Theta_{n,n+1}(\pi)^\vee \hookrightarrow \Ind_{P^\prime_{k,n+1-k}}^{G^\prime_{n+1}}(\tau \boxtimes \Theta_{n-k,n+1-k}(\pi_0)^\vee).
\]
By induction, we see that $\Theta_{n-k,n+1-k}(\pi_0)$ is irreducible and tempered, which imply that $\Ind_{P^\prime_{k,n+1-k}}^{G^\prime_{n+1}}(\tau \boxtimes \Theta_{n-k,n+1-k}(\pi_0)^\vee)$  is irreducible and tempered. Hence so is $\Theta_{n,n+1}(\pi)$.
\end{proof}

\end{document}